\documentclass[reqno]{amsart}%
\usepackage{amssymb,setspace,mathtools}
\usepackage[hmarginratio=1:1,a4paper]{geometry}
\usepackage{ifpdf}
\ifpdf
 \usepackage[hyperindex]{hyperref}%,pagebackref
\else
 \expandafter\ifx\csname dvipdfm\endcsname\relax
 \usepackage[hypertex,hyperindex]{hyperref}%
 \else
 \usepackage[dvipdfm,hyperindex]{hyperref}%
 \fi
\fi
\allowdisplaybreaks[4]
\numberwithin{equation}{section}
\theoremstyle{plain}
\newtheorem{thm}{Theorem}[section]

\theoremstyle{remark}

\DeclareMathOperator{\td}{d\mspace{-2mu}}
\DeclareMathOperator{\bell}{B}

\begin{document}

\title[Complete monotonicity of a ratio of gamma functions]
{A ratio of many gamma functions and its properties with applications}

\author[F. Qi]{Feng Qi}
\address[Qi]{Institute of Mathematics, Henan Polytechnic University, Jiaozuo 454010, Henan, China; College of Mathematics, Inner Mongolia University for Nationalities, Tongliao 028043, Inner Mongolia, China; School of Mathematical Sciences, Tianjin Polytechnic University, Tianjin 300387, China}
\email{\href{mailto: F. Qi <qifeng618@gmail.com>}{qifeng618@gmail.com}, \href{mailto: F. Qi <qifeng618@hotmail.com>}{qifeng618@hotmail.com}, \href{mailto: F. Qi <qifeng618@qq.com>}{qifeng618@qq.com}}
\urladdr{\url{https://qifeng618.wordpress.com}}

\author[W.-H. Li]{Wen-Hui Li}
\address[Li]{Department of Fundamental Courses, Zhenghzou University of Science and Technology, Zhengzhou 450064, Henan, China}
\email{\href{mailto: W.-H. Li <wen.hui.li@foxmail.com>}{wen.hui.li@foxmail.com}, \href{mailto: W.-H. Li <wen.hui.li102@gmail.com>}{wen.hui.li102@gmail.com}}
\urladdr{\url{https://orcid.org/0000-0002-1848-8855}}

\author[S.-B. Yu]{Shu-Bin Yu}
\address[Yu]{School of Mathematical Sciences, Tianjin Polytechnic University, Tianjin 300387, China}
\email{\href{mailto: S.-B. Yu <shubin.yu@qq.com>}{shubin.yu@qq.com}}

\author[X.-Y. Du]{Xin-Yu Du}
\address[Du]{School of Computer Science and Technology, Tianjin Polytechnic University, Tianjin 300387, China}
\email{\href{mailto: X.-Y. Du <aduadu1010@qq.com>}{aduadu1010@qq.com}}

\author[B.-N. Guo]{Bai-Ni Guo}
\address[Guo]{School of Mathematics and Informatics, Henan Polytechnic University, Jiaozuo 454010, Henan, China}
\email{\href{mailto: B.-N. Guo <bai.ni.guo@gmail.com>}{bai.ni.guo@gmail.com}, \href{mailto: B.-N. Guo <bai.ni.guo@hotmail.com>}{bai.ni.guo@hotmail.com}}
\urladdr{\url{http://www.researcherid.com/rid/C-8032-2013}}

\begin{abstract}
In the paper, the authors establish an inequality involving exponential functions and sums, introduce a ratio of many gamma functions, discuss properties, including monotonicity, logarithmic convexity, (logarithmically) complete monotonicity, and the Bernstein function property, of the newly introduced ratio, and construct two inequalities of multinomial coefficients and multivariate beta functions.
\end{abstract}

\keywords{ratio; gamma function; Bernstein function; completely monotonic function; logarithmically completely monotonic function; inequality; multinomial coefficient; multivariate beta function; logarithmic derivative; logarithmic convexity; integral representation; open problem}

\subjclass[2010]{Primary 26A48; Secondary 05A20, 26D07, 26A51, 26D15, 33B15, 44A10}

\thanks{This paper was typeset using \AmS-\LaTeX}

\maketitle
%\tableofcontents

\section{Preliminaries}

A real-valued function $f(x)$ defined on a finite or infinite interval $I\subseteq\mathbb{R}$ is said to be completely monotonic on $I$ if and only if $(-1)^kf^{(k)}(x)\ge0$ for all $k\in\{0\}\cup\mathbb{N}$ and $x\in I$.
A positive function $f(x)$ defined on a finite or infinite interval $I\subseteq\mathbb{R}$ is said to be logarithmically completely monotonic on $I$ if and only if $(-1)^k[\ln f(x)]^{(k)}\ge0$ for all $k\in\mathbb{N}$ and $x\in I$.
A nonnegative function $f(x)$ defined on a finite or infinity interval $I$ is called a Bernstein function if its derivative $f'(x)$ is completely monotonic on $I$.
In the paper~\cite{CBerg} and the monograph~\cite[pp.~66--68, Comments~5.29]{Schilling-Song-Vondracek-2nd}, it is pointed out that the terminology ``logarithmically completely monotonic function'' was explicitly defined in~\cite{compmon2, minus-one} for the first time. The logarithmically complete monotonicity is weaker than the Stieltjes function, but stronger than the complete monotonicity~\cite{CBerg, absolute-mon-simp.tex, JAAC384.tex}.
For more information on this topic, please refer to~\cite{Qi-Agar-Surv-JIA.tex, Schilling-Song-Vondracek-2nd, widder} and closely related references therein.
\par
Recall from~\cite[p.~51, (3.9)]{Temme-96-book} that the classical Euler gamma function $\Gamma(z)$ can be defined by
\begin{equation*}
\Gamma(z)=\lim_{n\to\infty}\frac{n!n^z}{(z)_{n+1}},
\end{equation*}
where $z\ne0,-1,-2,\dotsc$ and
\begin{equation*}%\label{rising-factorial-def}
(z)_n=\prod_{\ell=0}^{n-1}(z+\ell)
=
\begin{cases}
z(z+1)\dotsm(z+n-1), & n\ge1\\
1, & n=0
\end{cases}
\end{equation*}
for $z\in\mathbb{C}$ and $n\in\{0\}\cup\mathbb{N}$ is called the rising factorial. The logarithmic derivative $\psi(x)=[\ln\Gamma(z)]'=\frac{\Gamma'(z)}{\Gamma(z)}$ of the gamma function $\Gamma(z)$ and $\psi^{(k)}$ for $k\in\mathbb{N}$ are usually called in sequence the digamma function, the trigamma function, the tetragamma function, and the like.
\par
With the help of the gamma function $\Gamma(z)$, the binomial coefficient $\binom{m}{n}=\frac{m!}{n!(m-n)!}$ can be generalized as the multinomial coefficient
$$
\binom{a_1+a_2+\dotsm+a_m}{a_1,a_2,\dotsc,a_m}
=\frac{\Gamma\bigl(1+\sum_{i=1}^ma_i\bigr)} {\prod_{i=1}^{m}\Gamma(1+a_i)}
$$
and the classical beta function $\bell(a,b)=\frac{\Gamma(a)\Gamma(b)}{\Gamma(a+b)}$ can be generalized as the multivariate beta function
\begin{equation*}
\bell(a_1,a_2,\dotsc,a_m) =\frac{\Gamma(a_1)\Gamma(a_2)\dotsm\Gamma(a_m)} {\Gamma(a_1+a_2+\dotsm+a_m)},
\end{equation*}
where $\Re(a_1),\Re(a_2),\dotsc,\Re(a_m)\ge0$. See~\cite[Section~24.1.2]{abram} and~\cite[Section~II.2]{Flajolet-Sedgewick-2009}.

\section{Motivation}
Motivated by the papers~\cite{Ouimet-JMAA-2018, Alzer-CM-JMAA.tex} and related texts in the survey article~\cite{Qi-Agar-Surv-JIA.tex}, by establishing the inequality
\begin{equation}\label{Ouimet-lem-ineq}
\sum_{i=1}^{m}\frac1{x^{1/\nu_i}-1}+\sum_{j=1}^{n}\frac1{x^{1/\tau_j}-1}
>\sum_{i=1}^{m}\sum_{j=1}^{n}\frac1{x^{1/\mu_{ij}}-1},
\end{equation}
where $x>1$, $0<\mu_{ij}\le1$, $\nu_i=\sum_{j=1}^{n}\mu_{ij}$, $\tau_j=\sum_{i=1}^{m}\mu_{ij}$, and $\sum_{i=1}^{m}\nu_i=\sum_{j=1}^{n}\tau_j=1$, Ouimet obtained in~\cite{OUIMET-ARXIV-1907-05262} that the ratio
\begin{equation}\label{g(t)-Ouimet}
g(t)=\frac{\prod_{i=1}^{m}\Gamma(\alpha_it+1)\prod_{j=1}^{n}\Gamma\bigl(\beta_jt+1\bigr)} {\prod_{i=1}^{m}\prod_{j=1}^{n}\Gamma\bigl(\lambda_{ij}t+1\bigr)}
\end{equation}
is logarithmically completely monotonic function on $(0,\infty)$, where $m,n\in\mathbb{N}$ and $0<\alpha_i,\beta_j,\lambda_{ij}\le1$ such that
\begin{equation*}
\sum_{j=1}^{n}\lambda_{ij}=\alpha_i,\quad \sum_{i=1}^{m}\lambda_{ij}=\beta_j, \quad
\sum_{i=1}^{m}\alpha_i=\sum_{j=1}^{n}\beta_j
=\sum_{i=1}^{m}\sum_{j=1}^{n}\lambda_{ij}
=\sum_{j=1}^{n}\sum_{i=1}^{m}\lambda_{ij}\le1.
\end{equation*}
\par
We observe that
\begin{enumerate}
\item
the proof of the inequality~\eqref{Ouimet-lem-ineq} is lengthy and complicated;
\item
the inequality~\eqref{Ouimet-lem-ineq} can be refined and extended;
\item
the proof of logarithmically complete monotonicity of the ratio $g(t)$ in~\eqref{g(t)-Ouimet} is fatally wrong.
\end{enumerate}
\par
In this paper, we will
\begin{enumerate}
\item
refine and extend the inequality~\eqref{Ouimet-lem-ineq} and supply a concise proof of the refinement and extension;
\item
motivated by $g(t)$ in~\eqref{g(t)-Ouimet}, formulate a new ratio and prove its peroperties;
\item
construct inequalities of multinomial coefficients and multivariate beta functions.
\end{enumerate}

\section{A new inequality}

Now we present a new inequality which refines and extends the inequality~\eqref{Ouimet-lem-ineq}.

\begin{thm}\label{lambda>0-No-Restr}
Let $x>0$ and $\mu_{i,j}>0$ for $1\le i\le m$ and $1\le j\le n$. Then
\begin{equation}\label{Ouimet-lem-iq}
\sum_{i=1}^{m}\frac1{e^{x/\nu_i}-1}+\sum_{j=1}^{n}\frac1{e^{x/\tau_j}-1}
\ge2\sum_{i=1}^{m}\sum_{j=1}^{n}\frac1{e^{x/\mu_{ij}}-1},
\end{equation}
where $\nu_i=\sum_{j=1}^{n}\mu_{ij}$ and $\tau_j=\sum_{i=1}^{m}\mu_{ij}$ for $1\le i\le m$ and $1\le j\le n$.
\end{thm}

\begin{proof}
Recall from~\cite[p.~650]{Marshall-Olkin-Arnold} that
\begin{enumerate}
\item
a function $\varphi:[0,\infty)\to\mathbb{R}$ is said to be star-shaped if $\varphi(\nu x)\le\nu\varphi(x)$ for all $\nu\in[0,1]$ and $x\ge0$;
\item
a real function $\varphi$ defined on a set $S\subset\mathbb{R}^n$ is said to be super-additive if $x,y\in S$ implies $x+y\in S$ and $\varphi(x+y)\ge\varphi(x)+\varphi(y)$;
\item
if $\varphi:[0,\infty)\to\mathbb{R}$ is star-shaped, then $\varphi$ is super-additive;
\item
if $\varphi$ is a real function defined on $[0,\infty)$, $\varphi(0)\le0$, and $\varphi$ is convex, then $\varphi$ is star-shaped.
\end{enumerate}
\par
Let $h(x)=\frac{1}{e^x-1}$ for $x>0$. Then the inequality~\eqref{Ouimet-lem-iq} can be rearranged as
\begin{equation}\label{Exp-h(x)-Eq}
\sum_{i=1}^{m}h\biggl(\frac{x}{\nu_i}\biggr)+\sum_{j=1}^{n}h\biggl(\frac{x}{\tau_j}\biggr)
>2\sum_{i=1}^{m}\sum_{j=1}^{n}h\biggl(\frac{x}{\mu_{ij}}\biggr).
\end{equation}
A direct computation gives
\begin{gather*}
\frac{\td}{\td x}h\biggl(\frac1x\biggr)=\frac{e^{1/x}}{(e^{1/x}-1)^2 x^2},\quad
\frac{\td{}^2}{\td x^2}h\biggl(\frac1x\biggr)=\frac{e^{1/x} \bigl[e^{1/x}(1-2x)+2x+1\bigr]}{(e^{1/x}-1)^3 x^4},\\
\bigl[e^{1/x}(1-2x)+2x+1\bigr]'=2-e^{1/x}\biggl(\frac{1}{x^2}-\frac{2}{x}+2\biggr)\to0,\quad x\to\infty, \\
\bigl[e^{1/x}(1-2x)+2x+1\bigr]''=\frac{e^{1/x}}{x^4}>0,\quad
\lim_{x\to\infty}\bigl[e^{1/x}(1-2x)+2x+1\bigr]=0
\end{gather*}
for $x>0$. Consequently, combining these with $\lim_{x\to0^+}h\bigl(\frac1x\bigr)=0$ reveals that the function $h\bigl(\frac1x\bigr)$ is convex, then star-shaped, and then super-additive on $(0,\infty)$. As a result, it follows that
\begin{equation*}
\sum_{i=1}^{m}h\biggl(\frac{x}{\nu_i}\biggr)
=h\biggl(\frac{y}{\sum_{j=1}^{n}\mu_{ij}}\biggr)\ge \sum_{j=1}^{n}h\biggl(\frac{y}{\mu_{ij}}\biggr)\quad\text{and}\quad
\sum_{j=1}^{n}h\biggl(\frac{x}{\tau_j}\biggr)
=h\biggl(\frac{y}{\sum_{i=1}^{m}\mu_{ij}}\biggr)\ge \sum_{i=1}^{m}h\biggl(\frac{y}{\mu_{ij}}\biggr).
\end{equation*}
Substituting these two inequalities into the left hand side of~\eqref{Exp-h(x)-Eq} results in
\begin{gather*}
\sum_{i=1}^{m}h\biggl(\frac{x}{\nu_i}\biggr)
+\sum_{j=1}^{n}h\biggl(\frac{x}{\tau_j}\biggr)
\ge\sum_{i=1}^{m}\sum_{j=1}^{n}h\biggl(\frac{y}{\mu_{ij}}\biggr)
+\sum_{j=1}^{n}\sum_{i=1}^{m}h\biggl(\frac{y}{\mu_{ij}}\biggr)
\ge2\sum_{i=1}^{m}\sum_{j=1}^{n}h\biggl(\frac{y}{\mu_{ij}}\biggr).
\end{gather*}
The proof of the inequality~\eqref{Ouimet-lem-iq}, and then the proof of Theorem~\ref{lambda>0-No-Restr}, is thus complete.
\end{proof}

\section{A new ratio and its properties}

In this section, we formulate a new ratio of many gamma functions and find its properties.

\begin{thm}\label{Ouimet-thm-ext}
Let $\lambda=\begin{pmatrix}\lambda_{ij}\end{pmatrix}_{\substack{1\le i\le m\\ 1\le j\le n}}$ be a matrix such that $\lambda_{ij}>0$ for $1\le i\le m$ and $1\le j\le n$. Let $\alpha=(\alpha_1,\alpha_2,\dotsc,\alpha_m)$ and $\beta=(\beta_1,\beta_2,\dotsc,\beta_n)$ such that $\alpha_i=\sum_{j=1}^{n}\lambda_{ij}$ and $\beta_j=\sum_{i=1}^{m}\lambda_{ij}$ for $1\le i\le m$ and $1\le j\le n$.
Let $\rho\in\mathbb{R}$ and
\begin{equation}\label{f(m-n-lambda-alpha-beta-rho)(t)}
f_{m,n;\lambda,\alpha,\beta;\rho}(t)=\frac{\prod_{i=1}^{m}\Gamma(1+\alpha_it)\prod_{j=1}^{n}\Gamma\bigl(1+\beta_jt\bigr)} {\bigl[\prod_{i=1}^{m}\prod_{j=1}^{n}\Gamma\bigl(1+\lambda_{ij}t\bigr)\bigr]^\rho}.
\end{equation}
Then the following conclusions are valid:
\begin{enumerate}
\item
when $\rho\le2$, the second derivative $[\ln f_{m,n;\lambda,\alpha,\beta;\rho}(t)]''$ is a completely monotonic function of $t\in(0,\infty)$ and maps from $(0,\infty)$ onto the open interval
$$
\Biggl(0,\frac{\pi^2}{6}\Biggl(\sum_{i=1}^{m}\alpha_i^2 +\sum_{j=1}^{n}\beta_j^2 -\rho\sum_{i=1}^{m}\sum_{j=1}^{n}\lambda_{ij}^2\Biggr)\Biggr);
$$
\item
when $\rho=2$, the logarithmic derivative $[\ln f_{m,n;\lambda,\alpha,\beta;2}(t)]' =\frac{f_{m,n;\lambda,\alpha,\beta;2}'(t)} {f_{m,n;\lambda,\alpha,\beta;2}(t)}$ is a Bernstein function of $t\in(0,\infty)$ and maps from $(0,\infty)$ onto the open interval
\begin{equation*}
\left(0,\ln\frac{\prod_{i=1}^{m}\alpha_i^{\alpha_i}\prod_{j=1}^{n}\beta_j^{\beta_j}} {\Bigl(\prod_{i=1}^{m}\prod_{j=1}^{n}\lambda_{ij}^{\lambda_{ij}}\Bigr)^2}\right).
\end{equation*}
\item
when $\rho<2$, the logarithmic derivative $[\ln f_{m,n;\lambda,\alpha,\beta;p}(t)]'$ is increasing, concave, and from $(0,\infty)$ onto the open interval
$$
\Biggl(-\gamma(2-\rho)\sum_{i=1}^{m}\sum_{j=1}^{n}\lambda_{ij},\infty\Biggr),
$$
where $\gamma=0.57721566\dotsc$ is the Euler--Mascheroni constant;
\item
when $\rho=2$, the function $f_{m,n;\lambda,\alpha,\beta;2}(t)$ is increasing, logarithmically convex, and from $(0,\infty)$ onto the open interval $(1,\infty)$;
\item
when $\rho<2$, the function $f_{m,n;\lambda,\alpha,\beta;\rho}(t)$ has a unique minimum, is logarithmically convex, and satisfies
\begin{equation*}
\lim_{t\to0^+}f_{m,n;\lambda,\alpha,\beta;\rho}(t)=1 \quad\text{and}\quad \lim_{t\to\infty}f_{m,n;\lambda,\alpha,\beta;\rho}(t)=\infty.
\end{equation*}
\end{enumerate}
\end{thm}

\begin{proof}
Taking logarithm and differentiating give
\begin{align*}
\ln f_{m,n;\lambda,\alpha,\beta;\rho}(t)&=\sum_{i=1}^{m}\ln\Gamma(1+\alpha_it)+\sum_{j=1}^{n}\ln\Gamma\bigl(1+\beta_jt\bigr)
-\rho\sum_{i=1}^{m}\sum_{j=1}^{n}\ln\Gamma\bigl(1+\lambda_{ij}t\bigr),\\
[\ln f_{m,n;\lambda,\alpha,\beta;\rho}(t)]'&=\sum_{i=1}^{m}\alpha_i\psi(1+\alpha_it)+\sum_{j=1}^{n}\beta_j\psi\bigl(1+\beta_jt\bigr)
-\rho\sum_{i=1}^{m}\sum_{j=1}^{n}\lambda_{ij}\psi\bigl(1+\lambda_{ij}t\bigr),\\
[\ln f_{m,n;\lambda,\alpha,\beta;\rho}(t)]''&=\sum_{i=1}^{m}\alpha_i^2\psi'(1+\alpha_it)+\sum_{j=1}^{n}\beta_j^2\psi'\bigl(1+\beta_jt\bigr)
-\rho\sum_{i=1}^{m}\sum_{j=1}^{n}\lambda_{ij}^2\psi'\bigl(1+\lambda_{ij}t\bigr),
\end{align*}
and
\begin{gather*}
\lim_{t\to0^+}[\ln f_{m,n;\lambda,\alpha,\beta;\rho}(t)]''=\frac{\pi^2}{6}\Biggl(\sum_{i=1}^{m}\alpha_i^2+\sum_{j=1}^{n}\beta_j^2
-\rho\sum_{i=1}^{m}\sum_{j=1}^{n}\lambda_{ij}^2\Biggr),\\
\lim_{t\to0^+}[\ln f_{m,n;\lambda,\alpha,\beta;\rho}(t)]'= -\gamma(2-\rho)\sum_{i=1}^{m}\sum_{j=1}^{n}\lambda_{ij},\quad
\lim_{t\to0^+}f_{m,n;\lambda,\alpha,\beta;\rho}(t)=1.
\end{gather*}
Making use of the integral representation
\begin{equation*}
\psi^{(n)}(z)=(-1)^{n+1}\int_0^\infty\frac{t^n}{1-e^{-t}}e^{-zt}\td t, \quad \Re(z)>0
\end{equation*}
in~\cite[p.~260, 6.4.1]{abram} leads to
\begin{gather*}
[\ln f_{m,n;\lambda,\alpha,\beta;\rho}(t)]''=\sum_{i=1}^{m}\alpha_i^2\int_0^\infty\frac{s}{1-e^{-s}}e^{-(1+\alpha_it)s}\td s\\
+\sum_{j=1}^{n}\beta_j^2\int_0^\infty\frac{s}{1-e^{-s}}e^{-(1+\beta_jt)s}\td s -\rho\sum_{i=1}^{m}\sum_{j=1}^{n}\lambda_{ij}^2\int_0^\infty\frac{s}{1-e^{-s}}e^{-(1+\lambda_{ij}t)s}\td s\\
=\sum_{i=1}^{m}\alpha_i^2\int_0^\infty\frac{s}{e^s-1}e^{-\alpha_its}\td s
+\sum_{j=1}^{n}\beta_j^2\int_0^\infty\frac{s}{e^s-1}e^{-\beta_jts}\td s
-\rho\sum_{i=1}^{m}\sum_{j=1}^{n}\lambda_{ij}^2\int_0^\infty\frac{s}{e^s-1}e^{-\lambda_{ij}ts}\td s\\
=\sum_{i=1}^{m}\int_0^\infty\frac{u}{e^{u/\alpha_i}-1}e^{-tu}\td u
+\sum_{j=1}^{n}\int_0^\infty\frac{u}{e^{u/\beta_j}-1}e^{-tu}\td u
-\rho\sum_{i=1}^{m}\sum_{j=1}^{n}\int_0^\infty\frac{u}{e^{u/\lambda_{ij}}-1}e^{-tu}\td u\\
=\int_0^\infty u\Biggl(\sum_{i=1}^{m}\frac{1}{e^{u/\alpha_i}-1}
+\sum_{j=1}^{n}\frac{1}{e^{u/\beta_j}-1} -\rho\sum_{i=1}^{m}\sum_{j=1}^{n}\frac{1}{e^{u/\lambda_{ij}}-1}\Biggr)e^{-tu}\td u
\to0,\quad t\to0.
\end{gather*}
By virtue of Lemma~\ref{lambda>0-No-Restr}, when $\rho\le2$, we conclude that the second derivative $[\ln f_{m,n;\lambda,\alpha,\beta;\rho}(t)]''$ is completely monotonic with respect to $t\in(0,\infty)$.
\par
Since the second derivative $[\ln f_{m,n;\lambda,\alpha,\beta;\rho}(t)]''$ is completely monotonic with respect to $t\in(0,\infty)$, the logarithmic derivative $[\ln f_{m,n;\lambda,\alpha,\beta;\rho}(t)]'$ is increasing and concave on $(0,\infty)$. Hence,
\begin{gather*}
[\ln f_{m,n;\lambda,\alpha,\beta;\rho}(t)]' \ge\lim_{t\to0^+}\Biggl[\sum_{i=1}^{m} \alpha_i\psi(1+\alpha_it)+\sum_{j=1}^{n}\beta_j\psi\bigl(1+\beta_jt\bigr)
-\rho\sum_{i=1}^{m}\sum_{j=1}^{n}\lambda_{ij}\psi\bigl(1+\lambda_{ij}t\bigr)\Biggr]\\
=\sum_{i=1}^{m} \alpha_i\psi(1)+\sum_{j=1}^{n}\beta_j\psi(1) -\rho\sum_{i=1}^{m}\sum_{j=1}^{n}\lambda_{ij}\psi(1)
=\psi(1)\Biggl[\sum_{i=1}^{m} \alpha_i+\sum_{j=1}^{n}\beta_j-\rho\sum_{i=1}^{m}\sum_{j=1}^{n}\lambda_{ij}\Biggr]\\
=-\gamma(2-\rho)\sum_{i=1}^{m}\sum_{j=1}^{n}\lambda_{ij}
=\begin{dcases}
0,& \rho=2\\
-\gamma(2-\rho)\sum_{i=1}^{m}\sum_{j=1}^{n}\lambda_{ij}, & \rho<2
\end{dcases}
\end{gather*}
and
\begin{gather*}
[\ln f_{m,n;\lambda,\alpha,\beta;\rho}(t)]' \le\lim_{t\to\infty}\Biggl[\sum_{i=1}^{m} \alpha_i\psi(1+\alpha_it)+\sum_{j=1}^{n}\beta_j\psi\bigl(1+\beta_jt\bigr)
-\rho\sum_{i=1}^{m}\sum_{j=1}^{n}\lambda_{ij}\psi\bigl(1+\lambda_{ij}t\bigr)\Biggr]\\
=\sum_{i=1}^{m}\alpha_i\lim_{t\to\infty}[\psi(1+\alpha_it)-\ln(1+\alpha_it)]
+\sum_{j=1}^{n}\beta_j\lim_{t\to\infty}\bigl[\psi\bigl(1+\beta_jt\bigr)-\ln\bigl(1+\beta_jt\bigr)\bigr]\\
-\rho\sum_{i=1}^{m}\sum_{j=1}^{n}\lambda_{ij}\lim_{t\to\infty}\bigl[\psi\bigl(1+\lambda_{ij}t\bigr)-\ln\bigl(1+\lambda_{ij}t\bigr)\bigr]
+\ln\lim_{t\to\infty}\frac{\prod_{i=1}^{m}(1+\alpha_it)^{\alpha_i}\prod_{j=1}^{n}\bigl(1+\beta_jt\bigr)^{\beta_j}} {\prod_{i=1}^{m}\prod_{j=1}^{n}\bigl(1+\lambda_{ij}t\bigr)^{\rho\lambda_{ij}}}\\
=\ln\lim_{t\to\infty}\frac{\prod_{i=1}^{m}(1/t+\alpha_i)^{\alpha_i}\prod_{j=1}^{n}\bigl(1/t+\beta_j\bigr)^{\beta_j}} {\prod_{i=1}^{m}\prod_{j=1}^{n}\bigl(1/t+\lambda_{ij}\bigr)^{\rho\lambda_{ij}}}
+\ln\lim_{t\to\infty}t^{\sum_{i=1}^{m}\alpha_i+\sum_{j=1}^{n}\beta_j -\rho\sum_{i=1}^{m}\sum_{j=1}^{n}\lambda_{ij}}\\
=\ln\frac{\prod_{i=1}^{m}\alpha_i^{\alpha_i}\prod_{j=1}^{n}\beta_j^{\beta_j}} {\Bigl(\prod_{i=1}^{m}\prod_{j=1}^{n}\lambda_{ij}^{\lambda_{ij}}\Bigr)^\rho}
+\ln\lim_{t\to\infty}t^{(2-\rho)\sum_{i=1}^{m}\sum_{j=1}^{n}\lambda_{ij}}
=\ln\frac{\prod_{i=1}^{m}\alpha_i^{\alpha_i}\prod_{j=1}^{n}\beta_j^{\beta_j}} {\Bigl(\prod_{i=1}^{m}\prod_{j=1}^{n}\lambda_{ij}^{\lambda_{ij}}\Bigr)^\rho}
+\begin{dcases}
0, & \rho=2;\\
\infty, & \rho<2,
\end{dcases}
\end{gather*}
where we used the limit $\lim_{x\to\infty}[\psi(x)-\ln x]=0$ in~\cite[Theorem~1]{theta-new-proof.tex-BKMS} and~\cite[Section~1.4]{Sharp-Ineq-Polygamma-Slovaca.tex}.
Accordingly,
\begin{enumerate}
\item
when $\rho=2$, the logarithmic derivative $[\ln f_{m,n;\lambda,\alpha,\beta;2}(t)]'$ is positive and increasing and maps from $(0,\infty)$ onto
\begin{equation*}
\left(0,\ln\frac{\prod_{i=1}^{m}\alpha_i^{\alpha_i}\prod_{j=1}^{n}\beta_j^{\beta_j}} {\Bigl(\prod_{i=1}^{m}\prod_{j=1}^{n}\lambda_{ij}^{\lambda_{ij}}\Bigr)^2}\right).
\end{equation*}
\item
when $\rho<2$, the logarithmic derivative $[\ln f_{m,n;\lambda,\alpha,\beta;p}(t)]'$ is increasing, does not keep the same sign, and maps from $(0,\infty)$ onto
$$
\Biggl(-\gamma(2-\rho)\sum_{i=1}^{m}\sum_{j=1}^{n}\lambda_{ij},\infty\Biggr).
$$
\end{enumerate}
In conclusion, the logarithmic derivative $[\ln f_{m,n;\lambda,\alpha,\beta;2}(t)]'$ is a Bernstein function and the function $f_{m,n;\lambda,\alpha,\beta;\rho}(t)$ for $\rho<2$ has a minimum on $(0,\infty)$.
\par
It is easy to see that $\lim_{t\to0^+}f_{m,n;\lambda,\alpha,\beta;\rho}(t)=1$.
\par
In~\cite[p.~62, (3.20)]{Temme-96-book}, it was given that
\begin{equation*}
\ln\Gamma(z+1)=\biggl(z+\frac12\biggr)\ln z-z+\frac12\ln(2\pi)+\int_{0}^{\infty}\beta(t)e^{-zt}\td t,
\end{equation*}
where
\begin{equation*}
\beta(t)=\frac1t\biggl(\frac1{e^t-1}-\frac1t+\frac12\biggr).
\end{equation*}
Then a direct computation acquires
\begin{gather*}
\lim_{t\to\infty}\ln f_{m,n;\lambda,\alpha,\beta;\rho}(t)
=\lim_{t\to\infty}\Biggl(\sum_{i=1}^{m}\biggl[\biggl(\alpha_it+\frac12\biggr)\ln(\alpha_it)-\alpha_it +\frac12\ln(2\pi)+\int_{0}^{\infty}\beta(s)e^{-(\alpha_it)s}\td s\biggr]\\
+\sum_{j=1}^{n}\biggl[\biggl(\beta_jt+\frac12\biggr)\ln\bigl(\beta_jt\bigr)-\beta_jt +\frac12\ln(2\pi)+\int_{0}^{\infty}\beta(s)e^{-(\beta_jt)s}\td s\biggr]\\
-\rho\sum_{i=1}^{m}\sum_{j=1}^{n}\biggl[\biggl(\lambda_{ij}t+\frac12\biggr)\ln\bigl(\lambda_{ij}t\bigr)-\lambda_{ij}t +\frac12\ln(2\pi)+\int_{0}^{\infty}\beta(s)e^{-(\lambda_{ij}t)s}\td s\biggr]\Biggr)\\
=\frac{m+n-mn\rho}{2}\ln(2\pi)+\lim_{t\to\infty}\left(\ln\frac{\prod_{i=1}^{m}(\alpha_it)^{\alpha_it+1/2} \prod_{j=1}^{n}\bigl(\beta_jt\bigr)^{\beta_jt+1/2}} {\Bigl[\prod_{j=1}^{n}\prod_{j=1}^{n}\bigl(\lambda_{ij}t\bigr)^{\lambda_{ij}t+1/2}\Bigr]^\rho} -(2-\rho)t\sum_{i=1}^{m}\sum_{j=1}^{n}\lambda_{ij}\right)\\
=\frac{m+n-mn\rho}{2}\ln(2\pi)+\frac12\ln\frac{\prod_{i=1}^{m}\alpha_i\prod_{j=1}^{n}\beta_j} {\prod_{j=1}^{n}\prod_{j=1}^{n}\lambda_{ij}^\rho}
+\lim_{t\to\infty}\Biggl(\frac{m+n-mn\rho}2\ln t\\
+t\ln\frac{\prod_{i=1}^{m}\alpha_i^{\alpha_i} \prod_{j=1}^{n}\beta_j^{\beta_j}} {\Bigl(\prod_{j=1}^{n}\prod_{j=1}^{n}\lambda_{ij}^{\lambda_{ij}}\Bigr)^\rho}
+t(\ln t-1)(2-\rho)\sum_{i=1}^{m}\sum_{j=1}^{n}\lambda_{ij}\Biggr)
=\infty,
\end{gather*}
where, when $\rho=2$, we used the fact that
\begin{gather*}
\frac{\prod_{i=1}^{m}\alpha_i^{\alpha_i}\prod_{j=1}^{n}\beta_j^{\beta_j}} {\Bigl(\prod_{i=1}^{m}\prod_{j=1}^{n}\lambda_{ij}^{\lambda_{ij}}\Bigr)^2}
=\frac{\prod_{i=1}^{m}\alpha_i^{\alpha_i}}{\prod_{i=1}^{m}\prod_{j=1}^{n}\lambda_{ij}^{\lambda_{ij}}} \frac{\prod_{j=1}^{n}\beta_j^{\beta_j}}{\prod_{i=1}^{m}\prod_{j=1}^{n}\lambda_{ij}^{\lambda_{ij}}}\\
=\prod_{i=1}^{m}\frac{\alpha_i^{\alpha_i}}{\prod_{j=1}^{n}\lambda_{ij}^{\lambda_{ij}}} \prod_{j=1}^{n}\frac{\beta_j^{\beta_j}}{\prod_{i=1}^{m}\lambda_{ij}^{\lambda_{ij}}}
=\prod_{i=1}^{m}\frac{\prod_{j=1}^{n}\bigl(\sum_{\ell=1}^{n}\lambda_{i\ell}\bigr)^{\lambda_{ij}}} {\prod_{j=1}^{n}\lambda_{ij}^{\lambda_{ij}}}
\prod_{j=1}^{n}\frac{\prod_{i=1}^{m}\bigl(\sum_{\ell}^{m}\lambda_{\ell j}\bigr)^{\lambda_{j\ell}}} {\prod_{i=1}^{m}\lambda_{ij}^{\lambda_{ij}}}\\
=\prod_{i=1}^{m}\prod_{j=1}^{n}\Biggl(\frac{\sum_{\ell=1}^{n}\lambda_{i\ell}}{\lambda_{ij}}\Biggr)^{\lambda_{ij}} \prod_{j=1}^{n}\prod_{i=1}^{m}\Biggl(\frac{\sum_{\ell}^{m}\lambda_{\ell j}}{\lambda_{ij}}\Biggr)^{\lambda_{j\ell}}
>1\times1=1.
\end{gather*}
The proof of Theorem~\ref{Ouimet-thm-ext} is complete.
\end{proof}

\section{Two inequalities}

In this section, as did in~\cite[Sections~3 and~4]{Alzer-CM-JMAA.tex}, by applying the fourth conclusion in Theorem~\ref{Ouimet-thm-ext}, we derive two inequalities of multinomial coefficients $\binom{a_1+a_2+\dotsm+a_m}{a_1,a_2,\dotsc,a_m}$ and of multivariate beta functions $\bell(a_1,a_2,\dotsc,a_m)$.
\par
When $\rho=2$, the function $f_{m,n;\lambda,\alpha,\beta;\rho}(t)$ defined by~\eqref{f(m-n-lambda-alpha-beta-rho)(t)} can be rearranged as
\begin{align*}
f_{m,n;\lambda,\alpha,\beta;2}(t)&=\prod_{i=1}^{m}\frac{\Gamma\bigl(1+\sum_{j=1}^{n}\lambda_{ij}t\bigr)} {\prod_{j=1}^{n}\Gamma\bigl(1+\lambda_{ij}t\bigr)}
\prod_{j=1}^{n}\frac{\Gamma\bigl(1+\sum_{i=1}^{m}\lambda_{ij}t\bigr)} {\prod_{i=1}^{m}\Gamma\bigl(1+\lambda_{ij}t\bigr)}\\
&=\prod_{i=1}^{m}\binom{\sum_{j=1}^{n}\lambda_{ij}t}{\lambda_{i1}t,\lambda_{i2}t,\dotsc,\lambda_{im}t} \prod_{j=1}^{n}\binom{\sum_{i=1}^{m}\lambda_{ij}t}{\lambda_{1j}t,\lambda_{2j}t,\dotsc,\lambda_{mj}t}.
\end{align*}
For $a_i>0$ and $i\in\mathbb{N}$, multinomial coefficients and multivariate beta functions are connected by
\begin{equation*}
\binom{\sum_{i=1}^ma_i}{a_1,a_2,\dotsc,a_m}
=\frac{\sum_{i=1}^ma_i} {\prod_{i=1}^{m}a_i} \frac1{\bell(a_1,a_2,\dotsc,a_m)}.
\end{equation*}
Therefore, we have
\begin{align*}
f_{m,n;\lambda,\alpha,\beta;2}(t)&=\frac1{t^{2mn-m-n}} \prod_{i=1}^{m}\frac{\sum_{j=1}^{n}\lambda_{ij}} {\prod_{j=1}^{n}\lambda_{ij}} \frac1{\prod_{i=1}^{m}\bell(\lambda_{i1}t,\lambda_{i2}t,\dotsc,\lambda_{im}t)}\\
&\quad\times \prod_{j=1}^{n}\frac{\sum_{i=1}^{m}\lambda_{ij}}{\prod_{i=1}^{m}\lambda_{ij}} \frac1{\prod_{j=1}^{n}\bell(\lambda_{1j}t,\lambda_{2j}t,\dotsc,\lambda_{mj}t)}.
\end{align*}
\par
Let $\ell\in\mathbb{N}$ and $\theta_k\in(0,1)$ satisfy $\sum_{k=1}^{\ell}\theta_k=1$. Let $\lambda=\begin{pmatrix}\lambda_{ij}\end{pmatrix}_{\substack{1\le i\le m\\ 1\le j\le n}}$ be a matrix such that $\lambda_{ij}>0$ for $1\le i\le m$ and $1\le j\le n$.
By virtue of the fourth conclusion in Theorem~\ref{Ouimet-thm-ext}, the function $f_{m,n;\lambda,\alpha,\beta;2}(t)$ is logarithmically convex on $(0,\infty)$. Hence,
\begin{equation*}
f_{m,n;\lambda,\alpha,\beta;2}\Biggl(\sum_{k=1}^{\ell}\theta_ky_k\Biggr)\le \prod_{k=1}^{\ell}f_{m,n;\lambda,\alpha,\beta;2}^{\theta_k}(y_k).
\end{equation*}
Accordingly, by simplification, it follows that
\begin{gather*}
\frac{\prod_{j=1}^{n}\binom{\sum_{i=1}^{m}\lambda_{ij}\sum_{k=1}^{\ell}\theta_ky_k} {\lambda_{1j}\sum_{k=1}^{\ell}\theta_ky_k,\lambda_{2j}\sum_{k=1}^{\ell}\theta_ky_k,\dotsc, \lambda_{mj}\sum_{k=1}^{\ell}\theta_ky_k}}
{\prod_{k=1}^{\ell}\Bigl[\prod_{j=1}^{n}\binom{\sum_{i=1}^{m}\lambda_{ij}y_k} {\lambda_{1j}y_k,\lambda_{2j}y_k,\dotsc,\lambda_{mj}y_k}\Bigr]^{\theta_k}}\\
\le \frac{\prod_{k=1}^{\ell}\Bigl[\prod_{i=1}^{m}\binom{\sum_{j=1}^{n}\lambda_{ij}y_k} {\lambda_{i1}y_k,\lambda_{i2}y_k,\dotsc,\lambda_{im}y_k}\Bigr]^{\theta_k}} {\prod_{i=1}^{m}\binom{\sum_{j=1}^{n}\lambda_{ij}\sum_{k=1}^{\ell}\theta_ky_k} {\lambda_{i1}\sum_{k=1}^{\ell}\theta_ky_k, \lambda_{i2}\sum_{k=1}^{\ell}\theta_ky_k\,\dotsc, \lambda_{im}\sum_{k=1}^{\ell}\theta_ky_k}}
\end{gather*}
and
\begin{gather*}
\frac{\prod_{i=1}^{m}\bell(\lambda_{i1}\sum_{k=1}^{\ell}\theta_ky_k,\lambda_{i2}\sum_{k=1}^{\ell}\theta_ky_k,\dotsc, \lambda_{im}\sum_{k=1}^{\ell}\theta_ky_k)} {\prod_{k=1}^{\ell}\bigl[\prod_{i=1}^{m} \bell(\lambda_{i1}y_k,\lambda_{i2}y_k,\dotsc,\lambda_{im}y_k)\bigr]^{\theta_k}}\\ \times\frac{\prod_{j=1}^{n}\bell(\lambda_{1j}\sum_{k=1}^{\ell}\theta_ky_k,\lambda_{2j}\sum_{k=1}^{\ell}\theta_ky_k,\dotsc, \lambda_{mj}\sum_{k=1}^{\ell}\theta_ky_k)} {\prod_{k=1}^{\ell}\bigl[\prod_{j=1}^{n} \bell(\lambda_{1j}y_k,\lambda_{2j}y_k,\dotsc,\lambda_{mj}y_k)\bigr]^{\theta_k}}\\
\ge\frac{\prod_{k=1}^{\ell}y_k^{(2mn-m-n)\theta_k}}{\sum_{k=1}^{\ell}(\theta_ky_k)^{2mn-m-n}} \prod_{i=1}^{m}\frac{\sum_{j=1}^{n}\lambda_{ij}} {\prod_{j=1}^{n}\lambda_{ij}} \prod_{j=1}^{n}\frac{\sum_{i=1}^{m}\lambda_{ij}}{\prod_{i=1}^{m}\lambda_{ij}}
\prod_{k=1}^{\ell}\Biggl[\prod_{i=1}^{m}\frac{\prod_{j=1}^{n}\lambda_{ij}} {\sum_{j=1}^{n}\lambda_{ij}} \prod_{j=1}^{n}\frac{\prod_{i=1}^{m}\lambda_{ij}}{\sum_{i=1}^{m}\lambda_{ij}}\Biggr]^{\theta_k}.
\end{gather*}

\section{Three open problems}

Finally, we pose three open problems.

\subsection{First open problem}
The logarithmically complete monotonicity is stronger than the complete monotonicity~\cite{CBerg, absolute-mon-simp.tex, JAAC384.tex}. This means that a logarithmically completely monotonic function must be completely monotonic. Completely monotonic functions on the infinite interval $(0,\infty)$ have a characterization~\cite[p.~161, Theorem~12b]{widder}: a function $f(t)$ defined on the infinite interval $(0,\infty)$ is completely monotonic if and only if the integral
\begin{equation}\label{berstein-1}
f(t)=\int_0^\infty e^{-ts}\td\sigma(s)
\end{equation}
converges for $0<t<\infty$, where $\sigma(s)$ is nondecreasing.
In other words, a function $f(t)$ is completely monotonic on $(0,\infty)$ if and only if it is a Laplace transform of a nondecreasing measure $\sigma(s)$ on $(0,\infty)$.
\par
Under conditions of Theorem~\ref{Ouimet-thm-ext}, the second derivative $[\ln f_{m,n;\lambda,\alpha,\beta;\rho}(t)]''$ is completely monotonic with respect to $t\in(0,\infty)$. Motivated by the integral representation~\eqref{berstein-1}, we now pose the first open problem: can one find a closed expression of the nondecreasing measure $\sigma_{m,n;\alpha,\beta;\lambda}(s)$ such that
\begin{equation*}
[\ln f_{m,n;\lambda,\alpha,\beta;\rho}(t)]''=\int_0^\infty e^{-ts}\td\sigma_{m,n;\alpha,\beta;\lambda}(s)
\end{equation*}
converges for $0<t<\infty$?

\subsection{Second open problem}
Recall from~\cite[Theorem~3.2]{Schilling-Song-Vondracek-2nd} that
a function $f:(0,\infty)\to[0,\infty)$ is a Bernstein function if and only if it admits the representation
\begin{equation}\label{Levy-Khintchine-repreentation}
f(t)=a+bt+\int_0^\infty\bigl(1-e^{-ts}\bigr)\td\sigma(s),
\end{equation}
where $a,b\ge0$ and $\sigma(s)$ is a measure on $(0,\infty)$ satisfying $\int_0^\infty\min\{1,s\}\td\sigma(s)<\infty$. By Theorem~\ref{Ouimet-thm-ext}, the logarithmic derivative $[\ln f_{m,n;\lambda,\alpha,\beta;\rho}(t)]'$ is a Bernstein function on $(0,\infty)$. Motivated by the integral representation~\eqref{Levy-Khintchine-repreentation}, we now pose the second open problem: can one find the values of $a,b$ and present a closed expression of the measure $\sigma_{m,n;\alpha,\beta;\lambda}(s)$ such that
\begin{equation*}
[\ln f_{m,n;\lambda,\alpha,\beta;\rho}(t)]'=a+bt+\int_0^\infty\bigl(1-e^{-ts}\bigr)\td\sigma_{m,n;\alpha,\beta;\lambda}(s)
\end{equation*}
and $\int_0^\infty\min\{1,s\}\td\sigma_{m,n;\alpha,\beta;\lambda}(s)<\infty$ hold?
\par
In order to solve the above two open problems, we suggest readers to refer to the papers~\cite{CBerg, Berg-J-W-Lambert-2019, Berg-Pedersen-JCAM-2001, berg-pedersen-Alzer, Berg-Pedersen-ball-PAMS, berg-pedersen-rocky, Besenyei-MIA-13, Norlund-No-CM-JNT.tex, Mortici-monoburn.tex, Log-Poly-Prop-IIS.tex, Mort-Guo-Ma-Prad.tex, Catalan-Int-Surv.tex, Toader-Qi-M-B.tex, 22ICFIDCAA-Filomat.tex, Recipr-Sqrt-Geom-S.tex, Schroder-Seq-second.tex, Catalan-Number-S.tex, KMS-B14-0477.tex, Acta-Sinica-B20120547.tex, MIA-3303.tex, Qi-Zhang-Li-MJM-14} and closely related references therein.

\subsection{Third open problem}
Is the inequality~\eqref{Ouimet-lem-iq} in Theorem~\ref{lambda>0-No-Restr} sharp? Equivalently speaking, can the number $2$ in the right hand side of~\eqref{Ouimet-lem-iq} be replaced by a larger constant?

\end{document}